\documentclass[10pt]{amsart}
\usepackage{ulem}
\usepackage{array,youngtab}
\usepackage{color, graphicx, comment}
\usepackage{tikz}
\usepackage{ulem}
\usetikzlibrary{automata,positioning}

\usepackage{mathtools}

\definecolor{darkgreen}{rgb}{0.,0.5,0.}

\DeclareMathOperator{\CF}{CF}
\DeclareMathOperator{\se}{\mathbf{s}(\boldsymbol\varepsilon)}

\allowdisplaybreaks

\numberwithin{equation}{section} \overfullrule 5pt
\newtheorem{thm}{Theorem}[section]
\newenvironment{thmbis}[1]
  {\renewcommand{\thethm}{\ref{#1}$'$}%
   \addtocounter{thm}{-1}%
   \begin{thm}}
  {\end{thm}}
\newtheorem{prop}[thm]{Proposition}

\newtheorem{lem}[thm]{Lemma}

\theoremstyle{definition}

\newtheorem{eg}{Example}
\newtheorem{remark}{Remark}

\newcommand{\ff}{\mathbb{F}}

\title[
    ALGEBRAIC automatic CONTINUED FRACTIONS IN characteristic~$2$
\date{}
]
{
    ON A FAMILY OF $2$-AUTOMATIC SEQUENCES GENERATING ALGEBRAIC CONTINUED FRACTIONS IN characteristic~$2$}
\date{}

\author{Yining Hu}
\address[Yining Hu]{School of Mathematics and Statistics, 
Huazhong University of Science and Technology, Wuhan, PR China}
\email{huyining@protonmail.com}
\thanks{
 This work is supported by the National Science Foundation
of China grant 12001216.
}


\begin{document}
\begin{abstract}
	We present family of automatic sequences that define algebraic continued
	fractions in characteristic $2$. This family is constructed from 
	ultimately period words and contains the period-doubling sequence.
\end{abstract}

\maketitle
\section{Introduction}
%<<<
Ultimately periodic sequences are the simplest of all sequences. In this article
we study a family $\mathcal{S}$ of automatic sequences generated by ultimately periodic 
sequences and the continued fractions defined by them.

	 Each sequence $\bf{s}$ in $\mathcal{S}$ is built from  
	 an ultimately periodic sequence $(\varepsilon_n)_n=p_0\ldots p_{l-1} 
	 (e_0\ldots e_{d-1})^\infty$ in the following way.
   Starting from the empty word $W_0$, we consider the sequence of words 
	 $(W_n)_{n\geq0}$ such that $W_{n+1}=W_n,\varepsilon_n,W_n$ 
	 (note the comma is for concatenation and it will be omitted when it is 
	 suitable).
   Hence, we have $W_1=\varepsilon_0$, $W_2= \varepsilon_0,\varepsilon_1,\varepsilon_0$, 
	 etc. 
    Note that, for all $n\geq1$, $W_n$ is a palindrome, centred in $\varepsilon_n$,
		of length $2^n-1$.

	 By construction, the word $W_{n+1}$ starts by $W_n$ and 
	 and we define $\mathbf{s}(\boldsymbol\varepsilon)$ to be the limit of $(W_n)$.
	  Hence, we have
	 $$\mathbf{s}(\boldsymbol\varepsilon)=\varepsilon_0,\varepsilon_1,\varepsilon_0,\varepsilon_2,\varepsilon_0,\varepsilon_1,\varepsilon_0,\varepsilon_3,\varepsilon_0,\varepsilon_1,\varepsilon_0,\varepsilon_2,.... $$ 

	 We illustrate the construction of $\mathbf{s}(\boldsymbol\varepsilon)$ with two basic cases.

   1) $\boldsymbol\varepsilon=a,b,(c)^\infty$  then 
	 $$\mathbf{s}(\boldsymbol\varepsilon)=a,b,a,c,a,b,a,c,....=(a,b,a,c)^\infty.$$
   Here $\boldsymbol\varepsilon$ is ultimately constant  and 
	 $\mathbf{s}(\boldsymbol\varepsilon)$ is periodic. Consequently a basic property on continued 
	 fractions shows that $\CF(\mathbf{s}(\boldsymbol\varepsilon))$ is quadratic over $\ff_2(t)$.

   2) $\boldsymbol\varepsilon= (a,b)^\infty $   then 
	 $$\mathbf{s}(\boldsymbol\varepsilon)=a,b,a,a,a,b,a,b,a,b,...$$
	 Here  $\mathbf{s}(\boldsymbol\varepsilon)$ is the celebrated period-doubling sequence. 
   It has been proved that $\CF(s(\boldsymbol\varepsilon))$ satisfies an algebraic 
	 equation of degree $4$ (see \cite{Hu2022H, Hu2022L}).

Each sequence in $\mathcal{S}$ is $2$-automatic. The definition and relevant
results about automatic sequences are briefly recalled in Section \ref{ss:aut}.

Continued fractions defined by two classical example of automatic sequences 
in characteristic $2$ were studied in \cite{Hu2022H, Bugeaud2022H, Hu2022L},
where it was established that the Thue-Morse and period-doubling sequences
both define algebraic continued fractions of degree $4$. 
On the other hand, both sequences define quadratic formal power series. 
The situation is different from the real case:
Bugeaud \cite{Bugeaud2013} proved in 2013 that the continued
fraction expansion of an algebraic real number of degree at least $3$ is
not automatic.

Our family $\mathcal{S}$ is a natural generalization of the period-doubling 
sequence. We prove in this article that every sequence from this family 
defines an algebraic continued fraction in characteristic $2$.

In Section \ref{ss:1} we give the notation necessary for stating our
main result in \ref{ss:2}. 
In Section \ref{sec:proof} we prove the main theorem, and in Section 
\ref{sec:eg} we give three examples. Finally, in Section \ref{sec:ric} we
make a comment on the link between continued fractions and Riccati differential
equations.
%>>>

\subsection{Continued Fractions in Power Series Rings}\label{ss:1}
%<<<
Let $A=\{a_0,\ldots, a_k\}$ be a finite alphabet. We treat $a_j$ as formal
variables and define
\begin{align*}
	\ff_2[A]:=\ff_2[a_0,\ldots, a_k]\\
	\ff_2(A):=\ff_2(a_0,\ldots, a_k)\\
	\ff_2((A)):=\ff_2((\frac{1}{a_0},\ldots, \frac{1}{a_k})).
\end{align*}
Here	$\ff_2((\frac{1}{a_0},\ldots, \frac{1}{a_k}))$ denotes the ring of
power series of the form
\begin{equation}\sum_{n_0, \ldots, n_k \geq N} c_{n_0,n_1,\ldots, n_k} a_0^{-n_0}\cdots {a_k}^{-n_k},\label{eq:star}\end{equation}
where $N$ is an integer and $c_{n_0,n_1,\ldots, n_k}\in \ff_2$.

Let $(u_n)_{n\geq 0}$ be a sequence taking values in $A$. It defines a
formal power series $\sum_{n\geq 0} u_n z^n$ in $\ff_2[A][[z]]$.

We define a norm on $\ff_2((A))$ by assigning a series of the form 
\eqref{eq:star}
the number $2^{-m}$, where $m=\min\{ n_0+n_1 \cdots+ n_k \mid \ c_{n_0,n_1,
\ldots, n_k}\neq 0\}$ with the convention that $\min\emptyset=\infty$. This norm makes  $\ff_2((A))$ an ultrametric space.

The continued fraction $\CF(\mathbf{u})=[u_0,u_1,\ldots]$ is defined as 
the limit of the sequence $([u_0,u_1,\ldots, u_n])_n$:
\begin{align*}
	[u_0]&=u_0,\\
	[u_0, u_1, \ldots, u_{n}]&=u_0+\frac{1}{[u_1,\ldots, u_n]}\in \ff_2((A)),
\end{align*}
for $n\geq 1$. 
For example, 
\begin{align*}
	[u_0,u_1,u_2]&=u_0+\cfrac{1}{u_1+\frac{1}{u_2}}\\
	&=u_0+\frac{u_1^{-1}}{1+(u_1u_2)^{-1}}\\
	&=u_0+u_1^{-1}+u_1^{-2}u_2^{-1}+u_1^{-3}u_2^{-2}+\cdots
\end{align*}
The convergence of this sequence is proved in the same way
as in the case of classical contineud fraction for real numbers.

We may also choose to specialize the letters in $A$ to non-constant polynomials
in $\ff_2[z]$ (this gives a continuous map from $\ff_2((A))$ to $\ff_2((1/z))$),
and then $\CF({\mathbf{u}})$ can be seen as formal power series in $\ff_2((1/z))$.
	 For basic information on continued fractions, particularly in power series rings, the reader is refered to \cite{Lasjaunias2017}.

%>>>
\subsection{Main result}\label{ss:2}
%<<<
We prove the following theorem.

	 \begin{thm}\label{thm:main}
		 Let $\boldsymbol\varepsilon=p_0\ldots p_{l-1}(e_0\ldots e_{d-1})^\infty$ be an 
		 ultimately periodic sequence over the alphabet $A$.

		 (i) The continued fraction $\CF(\se)$ is algebraic of degree at
		 most $2^d$ over $\ff_2(A)$. It is of degree $2^d$ if $e_0$ is different
		 from the other letters.

		 (ii) The formal power series 
	$$
		 F:=\sum_{j\geq 0} s_j z^j \in \ff_2(A)[[z]].
	$$ 
		 is algebraic of degree at most $2^{d-1}$ over $\ff_2(A)(z)$. 
		 It is of degree $2^{d-1}$ if $e_0$ is different from the other letters.

	 \end{thm}
\begin{remark}
For the formal power series in 
in Theorem \ref{thm:main} and in the examples, we chose to keep the letters of 
an automatic sequence as formal variables instead of mapping them to elements
in $\ff_p$, because in the latter case, the formal power series and 
	consequently its minimal polynomial depend on the mapping.
	 For an automatic sequence $\bf{u}$ taking values $a_1,\ldots, a_k$, 
it is always possible treat the letters as formal variables and get 
	a minimal polynomial in $\ff_p[a_1,\ldots,a_k,x][y]$. This is because $F$ 
can be decomposed into a finite sum $\sum a_n F_n$, where each $F_n$ is algebraic
over $\ff_p(x)$. Therefore each $a_n F_n$ is algebraic over $\ff_p(a_1,\ldots,
a_k)(x)$ and so is $F$.
\end{remark}

	%>>>
	 
\subsection{Automatic sequences and algebraicity}\label{ss:aut}
%<<<
Here we recall briefly the definition and properties of automatic sequences.
We refer the readers to \cite{Allouche2003Sh} for a comprehensive exposition 
on the subject.
A sequence $(u_n)_n$ is said to be $k$-automatic if it can be generated by a
finite $k$-automaton. In this article we will use an equivalent definition.
The $k$-kernel of a sequence $(u_n)_n$ is defined as the smallest set
containing $(u_n)_n$ and stable under the operations $(v_n)_n\mapsto (v_{kn+r})_n$ for $r=0,1,\ldots, k-1$.
The sequence $\bf{u}$ is said to be $k$-automatic if its $k$-kernel is finite.

The well-known Theorem of Christol, Kamae, Mend\`es France and Rauzy 
\cite{Christol1980KMFR}
states that a formal power series in $\ff_p[[x]]$ is algebraic over 
$\ff_p(x)$ if and only if the sequence of its coefficients is $p$-automatic.
This theorem is generalized by Salon \cite{Salon1986} to multivariate cases:
We define the $k$-kernel of a  multi-dimensional sequence 
$\mathbf{u}=(u_{n_1,\ldots, n_d})_{n_1,\ldots,n_d}$ 
to be the smallest set containing $\bf{u}$ and stable under the operations 
$$(v_{n_1, \ldots, n_d})_{n_1,\ldots, n_d}\mapsto
(v_{k\cdot n_1 +r_1, \ldots, k\cdot n_d+r_d})_{n_1,\ldots, n_d}$$
for $0\leq r_1,\ldots, r_d <k$. The sequence $\bf{u}$ is said to be 
$k$-automatic if its $k$-kernel is finite. When $\bf{u}$ takes values in
$\ff_p$, then the formal power series $\sum u_{n_1,\ldots, n_d} x_1^{n_1}\cdot x_d^{n_d}$
is algebraic over $\ff_p(x_1,\ldots,x_d)$ if and only if $\bf{u}$ is 
$k$-automatic.

For $r=0,1,\ldots, p-1$, the Cartier operators $\Lambda_r$ are defined as
\begin{align*}
	\Lambda_r:\ff_p[[x]]&\rightarrow \ff_p[[x]]\\
	\sum_{n=0}^\infty u_n x^n &\mapsto \sum_{n=0}^\infty u_{pn+r} x^n.
\end{align*}
Thus for $f\in \ff_p[[x]]$, we have
\begin{equation*}
	f= \sum_{r=0}^{p-1} x^r (\Lambda_r f)^p.
\end{equation*}
In the multivariate case, the Cartier operators are
\begin{align*}
	\Lambda_{r_1,\ldots, r_d}:\ff_p[[x_1,\ldots, x_d]]&\rightarrow \ff_p[[x_1,\ldots, x_d]]\\
	\sum u_{n_1,\ldots, n_d} x_1^{n_1}\cdots x_d^{n_d} 
	&\mapsto \sum u_{pn_1+r_1, \ldots, pn_d+r_d} x_1^{n_1}\cdots x_d^{n_d}.
\end{align*}
For $f\in \ff_p[[x_1,\ldots, x_d]]$, we have
\begin{equation*}
	f= \sum_{r_1,\ldots,r_d=0}^{p-1} x_1^{r_1}\cdots x_d^{r_d} (\Lambda_{r_1,\ldots, r_d} f)^p.
\end{equation*}

The following lemma will be used in the proof of theorem \ref{thm:main}:
\begin{lem}\label{lem:deg}
	Let $\mathbf{u}=(u_{n_1,\ldots, n_d})_{n_1,\ldots, n_d}$ be a $p$-automatic sequence taking value in $\ff_p$. Let
	$\bf{v}$ be an element from the $p$-kernel of $\bf{u}$, then the degree of 
	algebraicity of 
	$\sum v_{n_1,\ldots, n_d} x_1^{n_1}\cdot x_d^{n_d}$ is less than or equal to that of 
		$\sum u_{n_1,\ldots, n_d} x_1^{n_1}\cdot x_d^{n_d}$.
\end{lem}
\begin{proof}
		Let $f=\sum u_{n_1,\ldots, n_d} x_1^{n_1}\cdot x_d^{n_d}$.
	We only need to prove that for all $(r_1,\ldots, r_d) \in \{0,1,\ldots, p-1\}^d$,
	for $\Lambda:=\Lambda_{r_1,\ldots, r_d}$, the degree of algebraicity of 
	$\Lambda f$ is less than or equal to 
	that of $f$.
	From the proof of Christol's theorem we know that there exists $n\geq 0$
	and $a_1, a_2, \ldots a_n\in \ff_p(x_1,\ldots, x_d)$ such that
	$$f=a_1 f^{p^1}+a_2 f^{p^2}+\cdots + a_n f^{p^n}.$$
	We make use of the following property of $\Lambda$:
	for two power series $g$ and $h$, we have 
	$$\Lambda (g\cdot h^p)=\Lambda (g) \cdot h,$$
	and 
	$$\Lambda \frac{g}{h}= \Lambda\frac{g\cdot h^{p-1}}{h^p}=\frac{\Lambda{(g\cdot
	h^{p-1}})}{h}.$$
	Therefore
	$$\Lambda f = \Lambda(a_1) f+\Lambda (a_2) f^p + \cdots +
	\Lambda(a_n) f^{p^{n-1}}\in \ff_p(x_1,\ldots, x_d)[f].$$
	We conclude that the degree of $\Lambda f$ is less than or equal to
	that of $f$.
\end{proof}
%>>>

\section{Proof of Theorem \ref{thm:main}}\label{sec:proof}
%<<<
	 To prove of the first part of theorem \ref{thm:main}, we consider 
	 the truncation 	 $[s_0,...,s_{2^n-2}]$ of the continued fraction 
	 $\CF(\se)$ for $n\geq 1$. This corresponds to the palindrome $W_n$.
	 For $n\geq1$, we set $u_n/v_n=[s_0,...,s_{2^n-2}]$. We have 
   $(u_1,v_1)=(s_0,1)$ and $(u_2,v_2)=(s_0s_1s_2+s_0+s_2,s_1s_2+1)$, etc. 
	 The first step of the proof was introduced in the author's previous work
	 with Lasjaunias concerning 
	 the particular case of the period-doubling sequence. In Lemma 3.2 from 
	 \cite{Hu2022L}, using
	 basic properties of continuants, we proved that the pair $(u_n,v_n)$ 
	 satisfies a simple recurrence relation. Indeed, for $n\geq1$, we have
	 \begin{equation}\label{eq:R}
		 u_{n+1}=\varepsilon_n u_n^2  \mbox{ and } v_{n+1}=\varepsilon_n u_n v_n+1  \tag{R}.
	 \end{equation}
	From \eqref{eq:R} we get immediately 
	$$\frac{v_{n+1}}{u_{n+1}}=\frac{v_n}{u_n} + \frac{1}{u_{n}}$$ and therefore we obtain
	 $$\frac{v_n}{u_n}=\sum_{1\leq k\leq n} \frac{1}{u_k}$$ 
	 for $n\geq1$. Consequently
	 \begin{equation}
		 \frac{1}{\CF({\mathbf{s(\boldsymbol\varepsilon)}})}=\lim _{n\rightarrow \infty} 
		 \frac{v_n}{u_n}= \sum_{n\geq 1} \frac{1}{u_n}.\tag{$\star$}
	 \end{equation}
It is easier to work with the following series which differs by a rational
fractional from $1/\CF(\se)$:
\begin{equation*}
	G:=\sum_{n>l} \frac{1}{u_n}.
\end{equation*}
	By regrouping every $d$ terms and using \eqref{eq:R}, we decompose $G$ as follows 
	\begin{equation} \label{eq:G}
		G=\sum_{n=0}^{d-1}\frac{ G_n^2}{e_n},
\end{equation}
where 
\begin{equation*}
	G_n=\sum_{k=0}^\infty \frac{1}{u_{l+n+kd}}
\end{equation*}
for $n=0,1,\ldots, d-1$.
It follows from their definition and \eqref{eq:R} that
$\{G_n\}$ satisfy the following relations
\begin{equation}
	G_0=\frac{1}{u_l}+\frac{G_{d-1}^2}{e_{d-1}}
\end{equation}
and 
	\begin{equation}\label{eq:gj}
		G_n= \frac{G_{n-1}^2}{e_{n-1}}
\end{equation}
for $n=1,2,\ldots, d-1$, with the convention $u_0=1$,
from which we deduce
	\begin{equation}\label{eq:g0}
	G_0=\frac{1}{u_l}+\frac{G_{0}^{2^d}}{e_{d-1}e_{d-2}^2\cdots e_0^{2^{d-1}}}.
\end{equation}

Using \eqref{eq:G}, \eqref{eq:gj} and \eqref{eq:g0}, we can express $G$ in terms of $G_0$
	\begin{equation*}
		G=\frac{1}{u_l}+G_0+\frac{G_0^2}{e_0}+\frac{G_0^4}{e_0^2 e_1}+\cdots+ 
		\frac{G_0^{2^{d-1}}}{e_0^{2^{d-2}}e_1^{2^{d-3}}\cdots e_{d-2}}.
	\end{equation*}
This means that the degree algebraicity of $G$ (and that of $\CF(\se)$) is at 
most that of $G_0$. This proves that the degree of algebraicity of 
$\CF(\se)$ is at most $2^d$.

When $e_0$ is different from all other letters,  $G_0$ is in the $2$-kernel of 
$G$, so that by lemma \ref{lem:deg}, the degree of $G$ is at least that
of $G_0$. 
To show that degree of algebraicity of $G_0$ is at least $2^d$, consider the
series
$$g=\sum_{k=0}^\infty \frac{1}{e_0^{2^{k\cdot d}}}$$
that satisfies 
$$e_0\cdot g^{2^d}+e_0\cdot g+1=0.$$
This is an irreducible polynomial by Eisenstein's criterion. 
If we project all the letters in $G_0$ to $e_0$, then we get the series
$e_0\cdot \sum_{k=l}^\infty \frac{1}{e_0^{2^{k\cdot d}}}$, which has the 
same degree of algebraicity as $g$. Since projection cannot increase degree, 
we conclude that in this case, the degree of algebraicity of $G_0$, and 
therefore that of $\CF(\se)$ is $2^d$.

%>>>

%<<<

\medskip
  To prove the second part of the theorem, we need the following lemma
%<<<
\begin{lem}\label{lem:pos}
	Suppose that the letters $e_j\; (0\leq j<d)$ are all distinct and that they are 
	different from the letters $p_j\; (0\leq j<l)$. For $j=0,1,\ldots, d-1$, define
	$$P_j:=\{ k\mid s_k=e_j\}.$$
	Then 
	\begin{equation}\label{eq:lem1 1}
	P_{j+1}=\{ 2k+1 \mid k\in P_j\}\quad \mbox{ for } j=0,1,\ldots, d-2
	\end{equation}
	and
	\begin{equation}\label{eq:lem1 2}
P_0=\{2k+1\mid k\in P_{d-1}\}\cup \{(2k+1)\cdot 2^l-1\mid k =1,2,\ldots\}.
	\end{equation}
	\end{lem}
	\begin{proof}
		To simplify notation, we write $e_n$ instead of $e_{n\pmod{d}}$ when $n\geq d$.
		By definition, $\se$ can be written in two ways:
		\begin{equation}\label{eq:lem1}
			W_l e_0 W_l e_1 W_l e_0 W_l e_2 W_l e_0 W_l e_1 W_l e_0 W_l\ldots
		\end{equation}
		and
		\begin{equation}\label{eq:lem2}
			W_{l+1} e_1 W_{l+1} e_2 W_{l+1} e_1 W_{l+1} e_3 W_{l+1} e_1 W_{l+1} e_2 W_{l+1} e_1 W_{l+1}\ldots
		\end{equation}
		Letter $e_0$ appears
		at the same places in \eqref{eq:lem1} as $e_1$ in \eqref{eq:lem2}, except
		that the length of the block $W_l$ is $2^l-1$, while that of $W_{l+1}$ is
		$2^{l+1}-1$. 
		By the assumption on $(e_j)$ and $(p_j)$, the letter $e_j$ does not 
		appear in $W_{l+j}$ $(0\leq j<d)$.
		The above observation allows us to conclude that 
  	$$P_{1}=\{ 2k+1 \mid k\in P_0\}.$$
		Similar argument shows that
	  $$P_{j+1}=\{ 2k+1 \mid k\in P_j\}\quad \mbox{ for } j=1,\ldots, d-2.$$

		To prove \eqref{eq:lem1 2},
		we also write $\se$ in two ways:
		\begin{equation}\label{eq:lem3}
			W_{l+d-1} e_{d-1} W_{l+d-1} e_{0} W_{l+d-1} e_{d-1} W_{l+d-1} e_{1} W_{l+d-1} e_{d-1}  W_{l+d-1} e_{0}\ldots % W_{l+d-1} e_{d-1} W_{l+d-1}\ldots
		\end{equation}
		and
		\begin{equation}\label{eq:lem4}
			W_{l+d} e_0 W_{l+d} e_{1} W_{l+d} e_0 W_{l+d} e_{2} W_{l+d} e_0 W_{l+d} e_{1} W_{l+d} e_0 W_{l+d}\ldots
		\end{equation}
		As with previous cases, the letter $e_0$ appears
		at the same places in \eqref{eq:lem4} as $e_{d-1}$ in \eqref{eq:lem3}, 
		except that the length of the block $W_{l+d}$ is $2^{l+d}-1$, 
		while that of $W_{l+d-1}$ is $2^{l+d-1}-1$. Unlike previous cases, the letter $e_0$ also occurs in $W_{l+d}$ at indices $(2k+1)\cdot 2^l-1$ for $k=1,2,\ldots$ That's how we get the additional term in \eqref{eq:lem1 2}.
	\end{proof}
	%>>>

  We consider another sequence
	$\boldsymbol\varepsilon'=p_0'\ldots p_{l-1}'(e_0'\ldots e_{d-1}')^\infty$.
	where the letters $p_j'$ $(0\leq j<l)$ and $e_j'$ $(0\leq j<d)$
	are all distinct, so that we can 
	apply lemma \ref{lem:pos} to obtain the sets $P_j$ for $j=0,1,\ldots, d-1$
	corresponding to the sequence $\mathbf{s}(\boldsymbol\varepsilon')$.

	We decompose the series $F$ as follows
	\begin{equation}\label{eq:sf0}
		F=R+ e_0 F_0+\cdots +e_{d-1} F_{d-1},
	\end{equation}
	where 
	$$R=\frac{\sum_{k=0}^{2^l-2}s_k z^k}{1+z^{2^l}}
	$$
	is a rational function which accounts for the preperiodic part of $\boldsymbol\varepsilon$
	and 
	$$F_n=\sum_{j\in P_n} z^j$$
	for $n=0,1,\ldots, d-1$. 

	From lemma \ref{lem:pos} we deduce the relations
	\begin{equation}\label{eq:sf1}
		F_n=zF_{n-1}^2
	\end{equation}
	for $n=1,\ldots,d-1$.
	On the other hand, since the supports of $R$ and $F_0, \ldots, F_{d-1}$ form a
	partition of $\mathbb{N}$, we have
	\begin{equation}\label{eq:sf2}
		F_0 =\frac{1}{1+z}-\frac{\sum_{k=0}^{2^l-2} z^k}{1+z^{2^l}}-
		\sum_{n=1}^{d-1} F_n
		=\frac{z^{2^l-1}}{1+z^{2^l}}- \sum_{n=1}^{d-1} F_n.
	\end{equation}
	From \eqref{eq:sf1} and \eqref{eq:sf2} we deduce
	\begin{equation}
		F_0 =\frac{z^{2^l-1}}{1+z^{2^l}}-\sum_{n=1}^{d-1} z^{2^{n}-1} F_0^{2^{n}}.
	\end{equation}
	Using \eqref{eq:sf0} and \eqref{eq:sf1} we can express $F$ using $F_1$ 
	\begin{equation}\label{eq:F}
	F=R+ \sum_{n=0}^{d-1} e_n z^{2^{n}-1} F_0^{2^n}.
	\end{equation}
	This shows that the degree of $F$ is at most that of $F_0$.

	The degree of $F_0$ is $2^{d-1}$: if $l$ were $0$, in place of $F_0$, by 
	\ref{eq:sf0}, we would have a series $f$ that satisfies
	$$f=\frac{1}{1+z}- \sum_{n=1}^{d-1} z^{2^{n}-1} f^{2^{n}}.$$
	Therefore the series $h:=zf$ satisfy
	$$(1+z)\cdot \sum_{n=0}^{d-1} h^{2^{n}}+z=0,$$
	which is an irreducible polynomial in $h$ by  Eisenstein's criterion.

	The support of $f$ is the positions of $e_0$ in 
	$$\mathbf{s}((e_0e_1\ldots e_{d-1})^\infty)=W_0 e_0 W_0 e_1 W_0 e_0 W_0 e_2 W_0 e_0 W_0 e_1 W_0 e_0 W_0\ldots$$
	Comparing the above expression with \eqref{eq:lem1} we obtain
	$$F_0=z^{2^l-1}f^{2^l}.$$
	Therefore $f$ is in the $2$-kernel of $F_0$. By lemma \ref{lem:deg}
	the degree of $F_0$ is at least that of $f$.

	When $e_0$ is different from the other letters, we can obtain $F_0$ from $F$ 
	mapping $e_0$ to $1$ and the other letters to $0$. This show that in this case
	the degree of $F$ is at least that of $F_0$.

%>>>
%>>>

\section{Examples}\label{sec:eg}
\begin{eg}
	%<<<
	Let $\boldsymbol\varepsilon=(ab)^\infty$. Then $\bf{s(\boldsymbol\varepsilon)}$ is the  period-doubling 
	sequence.
	We have 
	\begin{align*}
		\frac{1}{\CF(\mathbf{s})}=G&=\frac{1}{a}+\frac{1}{a^2b}+\frac{1}{a^4b^2a}+\frac{1}{a^8b^4a^2b}+\frac{1}{a^{16}b^8a^4b^2a}\cdots\\
		&=\frac{G_0^2}{a}+\frac{G_1^2}{b},
	\end{align*}
	with
	\begin{align*}
		G_0&=1+\frac{1}{a^2b}+\frac{1}{a^8b^4a^2b}+\cdots\\
		G_1&=\frac{1}{a}+\frac{1}{a^4b^2a}+\frac{1}{a^{16}b^8a^4b^2a}\cdots.
	\end{align*}
	The series $G_0$ and $\CF(\mathbf{s})$ satisfy
	\begin{equation*}
	G_0=1+\frac{G_0^4}{a^2b},
	\end{equation*}
	and
	\begin{equation*}
		\frac{1}{\CF(\mathbf{s})}=G=1+G_0+\frac{G_0^2}{a}.
	\end{equation*}
	We get as expected (see \cite{Hu2022L})
	$$(ab + b^2 + 1)+ (a^2b + ab^2)\cdot G+ ab\cdot  G^2+ G^4=0$$

	\medskip
	The formal power series defined by $\bf{s(\boldsymbol\varepsilon)}$ is
	$$F=a+bz+az^2+az^3+az^4+bz^5+az^6+\cdots=a F_0 +b F_1. $$
	The series $F_0$ and $F$ satisfy
	$$F_0=\frac{1}{1+z}-zF_0^2,$$
	$$F=aF_0 + b z F_0^2.$$
	We deduce
	$$(z^3 + z)\cdot F^2+  (az^2 + bz^2 + a + b)\cdot F+ 
	(a^2z + abz + b^2z + a^2 + ab)=0.$$
	%>>>
\end{eg}
\begin{remark}
	Here we have a simple example of an algebraic formal power series whose 
	continued fraction has unbounded partial quotients. If we let $a=b=x$
	in $G_0$, we get a series
	$$g=1+\frac{1}{x}+\frac{1}{x^5}+\frac{1}{x^{21}}+\cdots,$$
	which satisfies the algebraic equation
	$$\frac{g^4}{x}+g+1=0.$$
	It can be shown that
	$$g=[1,x,x^3,x,x^{11},x,x^3,x,x^{43},\ldots]=[1,x^{c_0},x^{c_1},x^{c_2},
	\ldots],$$ 
	where $c_{2n}=1$ and $c_{2n+1}=4c_n-1$ for all $n\geq 0$.
	The exponents of partial quotients are not bounded, but they form nontheless 
	a $2$-regular sequence (see \cite{Allouche1992Sh}).
\end{remark}
\begin{eg}
	Let $\boldsymbol\varepsilon=a(bc)^\infty$. 
	Then $\mathbf{s(\boldsymbol\varepsilon)}=abacabababacabacabac\cdots $ 
	We have 
	$$	\frac{1}{\CF(\mathbf{s})}=\frac{1}{a}+G,$$
	where
	\begin{align*}
		G&=\frac{1}{a^2b}+\frac{1}{a^4b^2c}+ \frac{1}{a^8b^4c^2b}+
		\frac{1}{a^{16}b^8c^4b^2c}+
		\frac{1}{a^{32}b^{16}c^8b^4c^2b}+\cdots\\
		&=	\frac{G_0^2}{b}+\frac{G_1^2}{c},
	\end{align*}
	with
	\begin{align*}
		G_0&=\frac{1}{a}+\frac{1}{a^4b^2c}+\frac{1}{a^{16}b^8c^4b^2c}+\cdots\\
		G_1&=\frac{1}{a^2b}+\frac{1}{a^8b^4c^2b}+\frac{1}{a^{32}b^{16}c^8b^4c^2b}+\cdots
	\end{align*}
	The series $G_0$, and $\CF(\mathbf{s(\boldsymbol\varepsilon)})$ satisfy
	\begin{align}
		G_0&=\frac{1}{a}+\frac{G_0^4}{b^2c},\label{eg:abcg0}\\
%		G&= \frac{1}{a}+G_0+\frac{G_0^2}{b},\\
		\frac{1}{\CF(\mathbf{s(\boldsymbol\varepsilon)})}&=G_0+\frac{G_0^2}{b}.\label{eg:abccf}
	\end{align}
	We deduce from  \eqref{eg:abcg0} and \eqref{eg:abccf} that 
	the minimal polynomial of $\CF(\mathbf{s})$ is
	\begin{equation}
	(ab^2c + abc^2 + c^2)x^4  + (a^2b^2c + a^2bc^2)x^3+ a^2bcx^2 + a^2.
	\end{equation}

	\medskip
	The formal power series defined by $\bf{s(\boldsymbol\varepsilon)}$ is
	\begin{align*}
		F&=a + bz + az^2 + cz^3 + az^4 + bz^5 + az^6 + bz^7 + az^8 + bz^9 + az^{10} +\cdots\\
		&=\frac{a}{1+z^2}+bF_0+cF_1.
	\end{align*}
	The series $F_0$ and $F$ satisfy
	$$F_0=\frac{z}{1+z^2}-zF_0^2,$$
	$$F= \frac{a}{1+z^2} +  bF_0 + c z F_0^2.$$
	From these we deduce that the minimal polynomial for $F$ is
	$$A x^2+Bx+C=0,$$
	where
	\begin{align*}
		A&=z^5 + z,\\
		B&=bz^4 + cz^4 + b + c,\\
		C&=b^2z^3 + bcz^3 + c^2z^3 + abz^2 + acz^2 + a^2z + b^2z + bcz + ab + ac.
	\end{align*}
\end{eg}
\begin{eg}
	Let $\boldsymbol\varepsilon=(aabb)^\infty$. Then $\mathbf{s(\boldsymbol\varepsilon)}=aaabaaabaaabaaaaaaab\ldots$.
	We have
	\begin{equation*}
		\frac{1}{	\CF(\mathbf{s(\boldsymbol\varepsilon)})}=G=
		\frac{1}{a}+\frac{1}{a^2a}+\frac{1}{a^4a^2b}+\frac{1}{a^8a^4b^2b}+
		\frac{1}{a^{16}a^8b^4b^2a}+ \frac{1}{a^{32}a^{16}b^8b^4a^2a}+\cdots
	\end{equation*}
	and
	$$G=\frac{G_0^2}{a}+\frac{G_1^2}{a}+ \frac{G_2^2}{b}+\frac{G_3^2}{b}.$$
	with 
	\begin{align*}
		G_0&=1+\frac{1}{a^8a^4b^2b}+\cdots\\
		G_1&=\frac{1}{a}+\frac{1}{a^{16}a^8b^4b^2a}+\cdots\\
		G_2&=\frac{1}{a^2a}+\frac{1}{a^{32}a^{16}b^8b^4a^2a}\cdots\\
		G_3&=\frac{1}{a^4a^2b}+\frac{1}{a^{64}a^{32}b^{16}b^8a^4a^2b}\cdots\\
	\end{align*}
	The series $G$ and $G_0$ satisfy
	\begin{align*}
		G_0&=1+\frac{G_0^{16}}{a^{12}b^3}\\
		G&=1+G_0+\frac{G_0^2}{a}+\frac{G_0^4}{a^3}+\frac{G_0^8}{a^6b}.
	\end{align*}
	We deduce
	$$G^{16}+c_8G^8+c_2G^2+c_1G+c_0=0 \label{eg2:G}.$$
	where
	\begin{align*}
		c_{0}&=a^{11}b^3 + a^{10}b^4 + a^3b^{11} + a^2b^{12} + a^6b^6 + a^4b^8 + a^2b^{10} + b^{12} + a^6b^2 +\\&\quad a^4b^4 + a^2b^6 + b^8 + 1,\\
		c_1&=a^{12}b^3 + a^{11}b^4 + a^4b^{11} + a^3b^{12},\\
		c_2&=a^{11}b^3 + a^{10}b^4 + a^8b^6 + a^6b^8 + a^4b^{10} + a^3b^{11},\\
		c_8&=a^6b^2 + a^4b^4 + a^2b^6,
	\end{align*}
	This polynomial is irreducible.

	\medskip
	The formal power series defined by $\mathbf{s(\boldsymbol\varepsilon)}$ is
	\begin{align*}
		F&=a(F_0+F_1)+a(F_2+F_3),
	\end{align*}
	where
	\begin{align*}
		F_0&=1 + z^{2} + z^{4} + z^{6} + z^{8} + z^{10} + z^{12} + z^{14} + z^{15} + z^{16} + z^{18}+\cdots\\
		F_1&=z + z^{5} + z^{9} + z^{13} + z^{17} + z^{21} + z^{25} + z^{29} + z^{31} + z^{33} + z^{37}+\cdots\\
		F_2&=z^{3} + z^{11} + z^{19} + z^{27} + z^{35} + z^{43} + z^{51} + z^{59} + z^{63} + z^{67}+\cdots\\
		F_3&=z^{7} + z^{23} + z^{39} + z^{55} + z^{71} + z^{87} + z^{103} + z^{119} + z^{127} + z^{135}  +  \cdots
	\end{align*}
	The series $F_0$ and $F$ satisfy
	\begin{align}
		F_0&=\frac{1}{1+z}+ z F_0^2+z^3F_0^4+z^7 F_0^8,\label{eg2:f0}\\
		F&=aF_0+azF_0^2+bz^3F_0^4+bz^7F_0^8.\label{eg2:f}
	\end{align}
	Here $F$ is not of degree $8$, but $4$, because we can write \eqref{eg2:f0}
	and \eqref{eg2:f} as
	\begin{equation}\label{eg3:1}
		\frac{1}{1+z}+ (F_0+z F_0^2)+z^3(F_0+zF_0^2)^4=0
	\end{equation}
	\begin{equation}\label{eg3:2}
		F=a(F_0+zF_0^2)+bz^3(F_0+zF_0^2)^4.
	\end{equation}
	From \eqref{eg3:1} and \eqref{eg3:2} we find minimal polynomial of
	$F$ to be
	$$c_4x^4+c_1x+c_0=0,$$
	where
	\begin{align*}
		c_0&=a^{4} z^{3} + a^{3} b z^{3} + a^{2} b^{2} z^{3} + a b^{3} z^{3} + b^{4} z^{3} + a^{4} z^{2} + a^{3} b z^{2} + a^{2} b^{2} z^{2} +\\ &\quad  a b^{3} z^{2} + a^{4} z + a^{3} b z + a^{2} b^{2} z + a b^{3} z + a^{4} + a^{3} b + a^{2} b^{2} + a b^{3}\\
		c_1&=a^{3} z^{4} + a^{2} b z^{4} + a b^{2} z^{4} + b^{3} z^{4} + a^{3} + a^{2} b + a b^{2} + b^{3}\\
		c_4&=z^{7} + z^{3}
	\end{align*}
	This example shows that the degree of the continued fraction is not 
	always twice that of the formal power series.
\end{eg}
%>>>

\section{Continued Fractions and Riccati Differential Equations}\label{sec:ric}
%<<<
%<<<
In 1977, Baum and Sweet \cite{Baum1977S}
considered considered the subset $D$ of irrational continued fractions 
$\alpha$ in $\mathbb{F}_2((1/t))$  such that
    $$\alpha=[0,a_1,a_2,\ldots,a_n,\ldots]$$ 
		%TODO ;
where $\deg(a_i)=1$ (i.e., $a_i=t$ or $t+1$) for $i\geq 1$. 
Consider the subset $P$ of $\mathbb{F}_2((1/t))$
$$P:=\{\alpha \in \mathbb{F}_2((1/t)) \mid \alpha=\sum_{i<0} a_i t^i\}.$$
Hence $P$ contains $D$. They proved the following theorem:
\begin{thm}\label{thm:bs}[Baum-Sweet, 1977]
	An element $\alpha\in P$ is in $D$ if and only if $\alpha$ satisfies
	$$\alpha^2+t\alpha+1=(1+t)\beta^2$$
	for some $\beta\in P$.
\end{thm}
This theorem has the following equivalent formulation (see \cite{Hu2022L}
for a proof):
\begin{thmbis}{thm:bs}
	An element $\alpha\in P$ is in $D$ if and only if $\alpha$ satisfies
	$$(\alpha\cdot t(t+1))'=\alpha^2+1.$$
\end{thmbis}

In \cite{Bugeaud2022H} and \cite{Hu2022L}, it was proved that the Thue-Morse 
and period-doubling continued fraction both satisfy the Riccati differential
equation
\begin{equation*}
	(ab(a+b) f)'=(ab)'(1+f^2).
\end{equation*}
We prove that these are in fact particular cases of a more general result:
%where is the code for verifying this?
\begin{prop}\label{riccati}
Let $\bf{u}$ be a sequence taking value in $\{a,b,a+b\}$, where $a$,
	$b$ and $a+b$ (if it is in the image of $\bf{u}$) are non-constant polynomials in $\mathbb{F}_2[t]$,  then the continued 
fraction 
	$$f=[u_0, u_1, \ldots]\in\mathbb{F}_2((1/t))$$
satisfies the Riccati differential equation 
$$(ab(a+b)f)'=(ab)'(1+f^2).$$

\end{prop}
%<<<
\begin{comment}
\begin{lem}
	Define
	$$F(P_n,Q_n):=ab(a+b)P_nQ_n+ab(P_n^2+Q_n^2).$$
	Then $F(P_n,Q_n)=ab+g_n^2$ for some $g_n\in \mathbb{F}_2[a,b]$
\end{lem}
\begin{proof}
	It is easy to check that the claim is true for $n=1,2$. Suppose now 
	$n>2$ and that the claim holds for $n-1$ and $n-2$.
	By the recurrence relations 
	$$P_n=u_nP_{n-1}+P_{n-2},\quad Q_n=u_nQ_{n-1}+Q_{n-2},$$
	\begin{align*}
		&\quad F(P_n,Q_n)\\
		&=u_n^2F(P_{n-1},Q_{n-1})+F(P_{n-2},Q_{n-2})+u_nab(a+b)(
	P_{n-1}Q_{n-2}+P_{n-2}Q_{n-1})
	\end{align*}
	Notice that $P_{n-1}Q_{n-2}+P_{n-2}Q_{n-1}$, being the determinant of 
	product of matrices each with determinant $1$, is equal to $1$ and
	if $u_n=a$ or $u_n=b$, $u_nab(a+b)=a^2b^2+u_n^2ab$, if $u_n=a+b$, 
 $u_nab(a+b)=u_n^2ab$.
	We have 
	$$F(P_n,Q_n)=ab+ u_n^2 g_{n-1}^2+ g_{n-2}^2+a^2b^2$$
	if $u_n=a$ or $u_n=b$, and
	$$F(P_n,Q_n)=ab+ u_n^2 g_{n-1}^2+ g_{n-2}^2$$
	if $u_n=a+b$.
\end{proof}
\end{comment}
%>>>

\begin{proof}[Proof of Proposition \ref{riccati}]
	Let $P_n/Q_n$ $(n\geq -1)$ be the convergents of $f$. Recall that
	$$(P_0,Q_0)=(u_0,1),\quad (P_{-1},Q_{-1})=(1,0)$$
	and 
	\begin{equation}\label{eq:pq}
	P_n=u_nP_{n-1}+P_{n-2},\quad Q_n=u_nQ_{n-1}+Q_{n-2}\quad \mbox{ for } n\geq 1.
	\end{equation}
	Then we only have to prove that 
	$$\lim_n (ab(a+b)f_n)'-(ab)'(1+f_n^2)=0,$$
	where $f_n=\frac{P_n}{Q_n}.$

	Define
	$$F_n:=ab(a+b)P_nQ_n+ab(P_n^2+Q_n^2),$$
	then
	$$(ab(a+b)f_n)'-(ab)'(1+f_n^2)=\frac{F_n'}{Q_n^2}.$$
	We prove by induction that for all $n\geq 0$,
	\begin{equation}F_n=ab+g_n^2\label{dag} \end{equation} for some 
	$g_n\in \mathbb{F}_2[a,b]$. 

	It is easy to check that the claim is true for $n=-1, 0$. Suppose now 
	$n>1$ and that the claim holds for $n-1$ and $n-2$.
	From the recurrence relations \eqref{eq:pq} we deduce
	\begin{align*}
		F_n =u_n^2F_{n-1}+F_{n-2}+u_nab(a+b)(
	P_{n-1}Q_{n-2}+P_{n-2}Q_{n-1}).
	\end{align*}
	Notice that $P_{n-1}Q_{n-2}+P_{n-2}Q_{n-1}=1$ by a basic property of 
	convergents, and
	if $u_n=a$ or $u_n=b$, $u_nab(a+b)=a^2b^2+u_n^2ab$, if $u_n=a+b$, 
 $u_nab(a+b)=u_n^2ab$.
	We have 
	$$F_n=ab+ u_n^2 g_{n-1}^2+ g_{n-2}^2+a^2b^2$$
	if $u_n=a$ or $u_n=b$, and
	$$F_n=ab+ u_n^2 g_{n-1}^2+ g_{n-2}^2$$
	if $u_n=a+b$.
	Thus we have proven \eqref{dag}.

	Therefore
	$$\frac{F_n'}{Q_n^2}= \frac{a'b+b'a}{Q_n^2}.$$
	Since the valuation of $1/Q_n$ in $1/t$ is at least $n$, the above
	term converges to $0$ in $\mathbb{F}_2((1/t))$ as $n$ goes to infinity.
\end{proof}
%>>>

\section{Acknowledgement}
I would like to thank Alain Lasjaunias (who has written a partial presentation on this subject intended for a private public) for valuable discussions.
I would also like to thank Jean-Paul Allouche for his help.

\bibliographystyle{plain}

\bibliography{article}

\end{document}